\documentclass{IEEEtran}
\usepackage{cite}
\usepackage{verbatim}
\usepackage{mathrsfs}
\usepackage{dsfont}
\usepackage{amsmath}
\usepackage{amsthm}
\usepackage{amssymb}
\usepackage{graphicx}
\usepackage{textcomp}
\def\BibTeX{{\rm B\kern-.05em{\sc i\kern-.025em b}\kern-.08em
    T\kern-.1667em\lower.7ex\hbox{E}\kern-.125emX}}
\newtheorem{definition}{Definition}
\newtheorem{theorem}{Theorem}
\newtheorem{lemma}{Lemma}
\newtheorem{corollary}{Corollary}

\newtheorem{remark}{Remark}
\newtheorem{example}{Example}
\newtheorem{assumption}{Assumption}

\usepackage{color}
\begin{document}
%\title{Nonlinear Consensus with Varying Topology: A Hilbert Metric Approach} 
\title{Consensus of A Class of Nonlinear Systems with Varying Topology: A Hilbert Metric Approach} 

\author{Dongjun Wu
\thanks{This work was supported in part by the European Research Council (ERC)
through the European Union’s Horizon 2020 Research and Innovation
Program under Grant 834142 (ScalableControl). }
 \thanks{The author is with Department of Automatic Control, Lund University, 22100 Lund, Sweden
 (email: dongjun.wu@control.lth.se).}
}
\maketitle
\begin{abstract}
%In this technical note, we present new results on consensus of continuous-time
%nonlinear multi-agent systems with varying topology. The results
%are proven utilizing non-Lyapunov-based methods, specifically by analyzing
%the system flow under the Hilbert metric. We demonstrate that this
%metric is highly flexible in studying consensus properties, and can 
%effectively
%handle nonlinearities and time dependencies. Notably, compared to
%existing works, regularity issues can be dealt with more efficiently
%using this metric. This work thus offers a new perspective for studying
%nonlinear consensus with varying topology.

				%In this technical note, we propose a new approach to study consensus of
				%continuous-time nonlinear systems based on Hilbert metric. 
				%We demonstrat that this metric is highly flexible in studying
				%consensus properties, and can effectively handle nonlinearities and time
				%dependencies. 
				%In particular, we relax some technical assumptions of some
				%standard results, while using much shorter proofs to prove stronger
				%results. Hence this approach provides new insights to nonlinear consensus with varying topology.
In this technical note, we introduce a novel approach to studying consensus of
continuous-time nonlinear systems with varying topology based on Hilbert metric. We demonstrate that this
metric offers significant flexibility in analyzing consensus properties, while effectively
handling nonlinearities and time dependencies. Notably, our approach relaxes key technical
assumptions from some standard results while yielding stronger conclusions with
shorter proofs. This framework provides new insights into nonlinear consensus under
varying topology.

\end{abstract}

\begin{IEEEkeywords}
Hilbert metric, nonlinear consensus, varying topology
\end{IEEEkeywords}

\section{Introduction }

Consensus under varying topology is a fundamental research question in multi-agent
and network dynamical systems \cite{Moreau2005}.
This topic has been extensively studied over the {last three decades
within} the control community. In the seminal work \cite{vicsek1995novel},
Vicsek {\it et al.} first observed that consensus (or coordination, as referred to in their study ) is
{closely} related
to the topology of multi-agent systems. Following this observation,
lots of efforts have been made to establish the theoretical foundations
for consensus with varying topology. Notable contributions include
the works of Jadbabaie et. al. \cite{Jadbabaie2003}, Moreau \cite{Moreau2005},
and Ren and Beard \cite{Ren2005TAC}, {among others}. These early 21st century
studies sparked a surge in research on consensus with varying topology
in subsequent years. 
{The theoretical developments in this area have found applications across various
domains,}
including power grids \cite{dorfler2012synchronization}, social networks
\cite{rainer2002opinion}, and formation control \cite{oh2015survey}. More recent
advancements can be found in works such as \cite{shi2013role},
\cite{anderson2016convergence}, and \cite{barabanov2018global}.
%Applications of these theories can be found in
%various fields such as power grids \cite{dorfler2012synchronization},
%social networks \cite{rainer2002opinion} and formation control \cite{oh2015survey}. More recent advancement can be found in for example
%\cite{shi2013role}, \cite{anderson2016convergence}, \cite{barabanov2018global}. 

%Consensus under varying topology is a fundamental research question in multi-agent and
%networked dynamical systems, driven by the structural changes inherent in these systems
%\cite{Moreau2005}. Over the past two decades, this topic has been extensively studied
%within the control community. The seminal work of Vicsek et al. \cite{vicsek1995novel}
%first observed that consensus (or coordination, as referred to in their study) is closely
%linked to the topology of multi-agent systems. This insight led to significant theoretical
%advancements, with notable contributions from Jadbabaie et al. \cite{Jadbabaie2003},
%Moreau \cite{Moreau2005}, and Ren and Beard \cite{Ren2005TAC}, among others. These early
%21st-century studies laid the foundation for a surge of research on consensus under
%dynamic topologies in subsequent years.

Most of the aforementioned works have primarily focused on linear
systems. However, nonlinearities are pervasive in multi-agent systems,
such as coupled oscillators \cite{acebron2005kuramoto}, mobile
robots \cite{dimarogonas2007rendezvous} and {social networks
\cite{rainer2002opinion}}. For these systems, the tools
for analyzing linear consensus are not directly applicable, necessitating
the use of nonlinear techniques. %
Various tools have been introduced to address the problem, including
dissipative analysis \cite{stan2007analysis,yao2009passivity}, monotone
system theory \cite{yu2011consensus,altafini2012consensus}, and non-smooth
techniques \cite{Lin2007,chung2009cooperative}. For further reviews
on this topic, the readers are referred to \cite{cao2012overview,ren2008distributed,dorfler2014synchronization}
and the references therein. 

As far as we know, the result obtained in \cite{Lin2007} stands among
the most general ones regarding nonliner consensus with varying topology.
Prior to this, Lin \emph{et al.} \cite{Lin2005}\emph{ }studied consensus
of a class of nonlinear systems with fixed topology using non-smooth
analysis techniques, with invariance principle being key of the proof.
They later extended their result to nonlinear systems with switching
topology \cite{Lin2007}, requiring the switching signal to be sufficiently
regular (the definition will be provided later). When switching is
present, invariance principle can no longer be used. To handle this,
a technical result from \cite{narendra1987persistent} was employed
in \cite{Lin2007}. It is worth mentioning that, even though the system
in consideration might possibly be smooth, non-smooth techniques had to be used \cite{Lin2005,Lin2007}.
\begin{comment}
Some \cite{angeli2006stability}, Angeli and Bliman further extended
Moreau's result to allow arbitrary time delays and relax convexity
of the allowed regions for the state transition map of each agent.
But the results were restricted to discrete time systems. 
\end{comment}
The result in \cite{Lin2007} can also be seen as a nonlinear extension
of the work \cite{Moreau2005}, as both works relied on the quasi-strongly
connectedness of multi-agent system as a key assumption.

The aim of this note is to relax some of the technical assumptions made in
\cite{Lin2007} and to provide new understandings for the problem.
%2) provide new criteria for nonlinear consensus with
%varying topology and 3) obtain stronger results under {weaker} regularity
%assumptions. 
%In particular, we remove the restrictive
%regularity assumption on the
%switching signal, which is important for example when the ``switching''
%is only continuous \cite{Moreau2004} or measurable.
%The proof techniques in \cite{Lin2007}
%{rely on} non-smooth analysis. 
{To achieve this goal, we adopt a novel approach, different from standard practice of
utilizing Lyapunov or non-smooth analysis, by}
leveraging the so-called Hilbert metric to analyze the system dynamics, {which proves to be
quite successful.}

Hilbert metric has already been used to study consensus.
For instance, in \cite{cortes2008distributed} Hilbert metric was
used to further relax and extend the results obtained in \cite{Moreau2004};
in \cite{Sepulchre2010}, it was used for consensus in non-commutative
spaces. However, to the best of our knowledge, it has not been introduced
to study nonlinear consensus with varying topology. This might be
related to the fact that the Hilbert metric was considered more suitable for linear
systems; e.g., one of the most widely used results -- Birkhoff's
Theorem \cite{birkhoff1957extensions} -- is applicable for linear
systems (or more generally, homogeneous systems). In this note, however,
we show that Hilbert metric can also serve as a systematic tool for studying
nonlinear consensus. In particular, it performs well in handling nonlinearities
and time dependencies. 

Before outlining the contributions of this note, it is necessary to
briefly recall the main result obtained in \cite{Lin2007}. Consider
the following multi-agent system with switching topology controlled
by the switching signal $\sigma:\mathbb{R}_{\ge0}\to\{1,\cdots,N\}$:
\begin{equation}
\begin{cases}
\dot{x}_{1}=f_{\sigma(t)}^{1}(x_{1},\cdots,x_{n})\\
\vdots\\
\dot{x}_{n}=f_{\sigma(t)}^{n}(x_{1},\cdots,x_{n})
\end{cases}\label{sys:lin}
\end{equation}
where $x_{i}\in\mathbb{R}^{m}$. For each $p\in\{1,\cdots,N\}$, we
associate a digraph $\mathcal{G}_{p}$ with the system defined by
vector fields $\{f_{p}^{i}\}_{i=1}^{n}$. $\mathcal{G}_{p}$ has $n$
vertices denoted as $\{1,\cdots,n\}$, and a link $(i,j)$ is in $\mathcal{G}_{p}$
if $f_{p}^{i}$ depends explicitly on $x_{j}$. The digraph is called
\emph{quasi strongly connected (QSC) }if there exists a node $k$
such that for each node $j\ne k$, there is a directed path joining
node $k$ to node $j$. Such a node is called a \emph{center}. In
other words, a digraph is QSC if there is a directed spanning tree,
and the root of the tree is a center. For a switching graph corresponding
to the system (\ref{sys:lin}), we say the graph is \emph{uniformly
quasi strongly connected (UQSC) }if there exists a constant $T>0$
such that the union of the graph on the interval $[t,t+T]$ for any
$t$ is QSC. Let $\mathcal{C}_{i}$ be the polytope formed by $x_{i}$
and its neighboring agents, and $T_{x_{i}}\mathcal{C}_{i}$ the tangent
cone of $\mathcal{C}_{i}$ at $x_{i}$ in $\mathbb{R}^{m}$, see \cite{Lin2007}
for more detailed explanations of these terminologies. 

\begin{comment}
The following is the key assumption in \cite{Lin2007}:
\textbf{Assumption A0:} for each $p\in\{1,\cdots,n\}$, 
\begin{enumerate}
\item $f_{p}^{i}$ is locally Lipschitz and $f_{p}^{i}$ is in the relative
interior of the cone $T_{x_{i}}\mathcal{C}_{i}$, or $f_{p}^{i}\in{\rm ri}(T_{x_{i}}\mathcal{C}_{i})$; 
\item the switching signal is piece-wise constant with minimum dwell time
$\tau_{D}$; 
\end{enumerate}
\end{comment}

One of the main results in \cite{Lin2007} can be restated as: 
\begin{theorem}[Lin \textit{et al}. \cite{Lin2007}]
\label{thm:lin}Consider the system \eqref{sys:lin}. Assume that
for each $p\in\{1,\cdots, {N}\}$: 1) $f_{p}^{i}$ is locally Lipschitz
and $f_{p}^{i}$ is in the relative interior of the cone $T_{x_{i}}\mathcal{C}_{i}$,
i.e., $f_{p}^{i}\in{\rm ri}(T_{x_{i}}\mathcal{C}_{i})$; 2) the switching
signal is piece-wise constant with minimum dwell time $\tau_{D}$;
then the system achieves global consensus \emph{
if and only if} the system (\ref{sys:lin}) is UQSC. 
\end{theorem}
A switching signal satisfying the assumptions (having a minimum dwell
time $\tau_{D}$) in Theorem \ref{thm:lin} is said to be {\emph regular}.

Note that under the assumption of Theorem \ref{thm:lin}, the vector
field $f_{p}^{i}$ can be written as $f_{i}^{p}=\sum a_{ij}^{p}(x)(x_{j}-x_{i})$
for some non-negative scalar functions $a_{ij}^{p}$. 
%But regularity
%of $a_{ij}$ is unknown - it might be discontinuous.
Equivalently,
this means that the system can be written as $\dot{x}=A_{p}(x)x$
with $(A_{p}(x))_{ij}=a_{ij}^{p}(x)$, in which the matrix $A_{p}(x)$
is Metzler, that is, $A_{ij}\ge0$ for all $i\ne j$, and each
row sums to zero. System of the form $\dot{x}=A(x)x$ with $A(x)$
being Metzler has been recently considered by Kawano and Cao \cite{Kawano2022},
where such systems were referred to as ``virtually positive''. Due
to the special structure $\dot{x}=A_{\sigma(t)}(x)x$, consensus of
such systems shares a lot in common with linear time varying multi-agent
systems. A remarkable result concerning the latter was obtained by
Moreau in \cite{Moreau2004}: 
\begin{theorem}[Moreau \cite{Moreau2004}]
\label{thm:moreau} Consider the LTV system $\dot{x}=A(t)x$. Assume
that $A(\cdot)$ is uniformly bounded and piecewise continuous. Assume
that, for every time $t$, $A(t)=(a_{ij}(t))\in\mathbb{R}^{n\times n}$
is Metzler with zero row sums. If there exists an index $k\in\{1,\cdots,n\}$,
a threshold value $\delta>0$ and an interval length $T>0$ such that
for all $t\ge0$, 
\begin{equation}
\int_{t}^{T+t}a_{ik}(s)ds\ge\delta,\quad\forall i\in\{1,\cdots,n\}\backslash\{k\},\label{eq:moreau}
\end{equation}
then the system achieves exponential consensus.
\end{theorem}
\begin{comment}
When a graph whose associated graph $A(t)$ satisfies the above properties,
we say that the graph is \emph{$\delta$-connected}.
\end{comment}
The proof of Theorem \ref{thm:moreau} provided in \cite{Moreau2004}
was based on Lyapunov analysis, see also \cite{Jadbabaie2003,Moreau2005} for
discrete time versions. In particular, separable Lyapunov functions
were used. This is a technique commonly used in monotone systems,
see for example \cite{dirr2015separable,rantzer2015scalable,feyzmahdavian2017stability,Tsitsiklis1986}. 

Theorem \ref{thm:lin} and Theorem \ref{thm:moreau} and share
the same spirit. Nevertheless, their proof techniques were quite different: 
the former relies on on non-smooth analysis while the latter on Lyapunov analysis.
{It is then a question whether the two results can
be understood in a unifying framework}. This has remained an open question. In
this note, we provide an affirmative answer to it.

\begin{comment}
As we will see in the next section, the result that we are going to
present for nonlinear consensus with varying topology is analogous
to Theorem \ref{thm:moreau}. But we underscore that, due to the presence
of nonlinearity, it is not clear how to extend the proof techniques
to handle this new challenge.
\end{comment}
{} 

We address the problem by a non-Lyapunov-based method, i.e., through
the analysis of the evolution of the system dynamics under the Hilbert
metric. The contributions of the note are twofold:
\begin{enumerate}
		\item {Utilizing Hilbert metric to analyze nonlinear consensus is new.
						It does not rely on Lyapunov \cite{Moreau2004} or 
						non-smooth analysis \cite{Lin2007}. Our proof based on Hilbert
						metric for consensus is largely simplified compared to existing
						works. This has led to new understandings and insights of nonlinear
						consensus.
						%Moreover, weaker regularity assumptions are required, while
		%stronger results can be obtained; e.g., \cite{Lin2007} proves only
		%asymptotic consensus, but our analysis shows exponential consensus in certain
		%situations.
}

\item %The system to be studied is more general than (\ref{sys:lin}):
		{Instead of regular switching topology considered in \cite{Lin2007}, we
				are able to analyze more general time varying topologies, requiring 
				only measurability of the switching. This
				includes cases such as piece-wise continuous switching studied in \cite{Moreau2004}.
				Consequently, our results extend classical findings, including Theorem \ref{thm:lin} and
				\ref{thm:moreau}, by relaxing key technical assumptions. Furthermore,
				we obtain stronger results; for instance, while \cite{Lin2007} establishes only
				asymptotic consensus, our analysis demonstrates exponential consensus under
				mild additional assumptions.
		}

%\item New criteria for consensus are proposed, which are related to the
%accumulated graph of the multi-agent system and some of them 
%are easy to check.

\end{enumerate}

Organization of the paper: In Section \ref{sec:pre}, we provide 
the problem setting and prove some technical results. In particular,
we give some explicit formulas and new definitions which will be
crucial for the next section. Section \ref{sec:mainresult} contains
the main results of this note and Section \ref{sec:simu} demonstrate
a simulation result. 

\emph{Notations: }$|\cdot|$ stands for Euclidean $2$-norm. {For a
dynamical system, use $\phi(t,t_{0},x_{0})$ to represent the solution
at $t$ from initial state $(t_{0},x_{0})$.} The interior of a set
$S$ is denoted ${\rm Int}S$. Being $X$, $Y$ some topological spaces,
denote $C(X,Y)$ the space of continuous maps from $X$ to $Y$.
Given a set $S\subseteq X$,
the indicator function $1_{S}:X\to\{0,1\}$ is defined to be $1_{S}(x)=1$
if $x\in S$ and $0$ otherwise. For a Metzler matrix $A$ with row sum zero, $\mathcal{G}^A$ 
represents the graph associated with $A$. Given $x,y$, $d(x,y)$ stands
for the Hilbert metric between $x$ and $y$. $\mathbb{N}_+$ is the set of positive natural numbers.
\begin{comment}
$E(\mathcal{G})$ stands for the edge set of graph $\mathcal{G}$. 
\end{comment}
$\mathds{1}_{n}$ a column vector of dimension $n$ with all ones. 
{A continuous function $\beta: [0,a)\times [0,\infty)\to [0,\infty)$ is called a class
		$\mathcal{KL}$ function if 1) $r\mapsto \beta(r,s)$ is strictly increasing and
		$\beta(0,s) \equiv 0$ for all $s\ge 0$; 2) $s\mapsto \beta(r,s)$ is decreases to
		$0$ as $s\to \infty$.
}

\section{Preliminary Results} \label{sec:pre}

\subsection{Problem setting}

We consider continuous time multi-agent nonlinear systems of the form
\begin{equation}
\dot{x}_{i}=\sum_{j=1}^{n}a_{ij}(t,x)(x_{j}-x_{i}),\label{sys:cons}
\end{equation}
for $i=1,\cdots,n$, where the state of each agent is in $\mathbb{R}$
and $t\mapsto a_{ij}(t,\cdot)$ is measurable; in addition, $a_{ij}(t,x)\ge0$
for all $t\ge0$, $x\in\mathbb{R}^{n}$ and $i,j\in\{1,\cdots,n\}$,
$i\ne j$.

\begin{remark} \label{rmk:1}
We can also consider more general systems having the following form
\begin{equation}
\dot{x}_{i}=a(t)x_{i}+b(t)+\sum_{j=1}^{n}a_{ij}(t,x)(x_{j}-x_{i})\label{eq:inter-dyn}
\end{equation}
where $a(\cdot)$ and $b(\cdot)$ are bounded on $\mathbb{R}_{+}$.
Indeed, define $y(t)=e^{\int_{0}^{t}a(s){\rm d}s}x-\left(\int_{0}^{t}e^{\int_{0}^{\tau}a(s){\rm d}s}b(\tau){\rm d}\tau\right)\mathds{1}$
we have $\dot{y}=\tilde{A}(t,y)y$ where $\tilde{A}$ is Metzler and
satisfies $\tilde{A}\mathds{1}=0$. %
\begin{comment}
%
\begin{remark}
\label{rmk:a} The non-positivity assumption of $a(\cdot)$ can be
relaxed. For example, we may assume that there exists some $T>0$
such that 
\begin{equation}
\int_{t_{0}}^{t_{0}+T}a(t){\rm d}t\le0\label{eq:a-assump}
\end{equation}
for all $t_{0}\ge0$. e.g., $a(t)$ is periodic with period $T$ and
that $\int_{0}^{T}a(t){\rm d}t\le0$. %When the inequality \eqref{eq:a-assump} becomes strict, it can be
 %shown that $x_i (t)\to 0$ for all $i$ exponentially by arguing that
 %the matrix measure $\mu_\infty$
 %This amounts 
 %to a 
 %We will come 
 %back to this in the proof of Theorem \ref{thm:lin}.
\end{remark}
\end{comment}
 %
\begin{comment}
In fact, we can define a new state called $y=(y_{1},\cdots,y_{n})^{\top}$
through $y_{i}(t)=x_{i}(t)-\int_{0}^{t}e^{a(\tau)(t-\tau)}b(\tau){\rm d}\tau.$
It is then readily check that $\dot{y}_{i}=a(t)y_{i}+\sum_{i=1}^{n}a_{ij}(t,y)(y_{j}-y_{i}).$
Obviously, consensus of $y_{i}$ is equivalent to that of $x_{i}$.
Therefore, without loss of generality, we can assume $b(\cdot)\equiv0$
in \eqref{eq:inter-dyn}. 
\end{comment}
\end{remark}
{
		We mention a few practical examples that can be written as \eqref{sys:cons}.
\begin{itemize}
	\item \textbf{Kuramoto oscillators:} This is a widely studied nonlinear consensus
			model with applications in various engineering domains. Consider a network of
			Kuramoto oscillators with time-varying and state-dependent coupling:
    \begin{equation*}
    \dot{\theta}_{i}=\omega_{i}(t)+\sum_{j}k_{ij}(t,\theta)\sin(\theta_{i}-\theta_{j}) %\label{eq:kuramoto1}
    \end{equation*}
	where $\omega_{i}(t)$ are the natural frequencies, and $k_{ij}(t,\theta)$ are
	non-negative coupling functions. Phase synchronization occurs when
	$|\theta_{i}(t)-\theta_{j}(t)| \to 0$ as $t \to \infty$. This can only be achieved if
	all natural frequencies are identical. In this case, the model can be transformed into
	the standard form (\ref{sys:cons}) by defining $a_{ij}(t,\theta):= k_{ij}(t,\theta)
	\frac{\sin(\theta_i -\theta_j)}{\theta_i - \theta_j}$.
	Generalizations of the Kuramoto oscillator that can be expressed in the form of
	(\ref{sys:cons}) can be found in \cite{dorfler2014synchronization}.

    \item \textbf{Cucker-Smale model:} This model is commonly used to describe flocking behavior \cite{cucker2007emergent}:
    \begin{align*}
    \dot{x}_{i} & =v_{i}\\
    \dot{v}_{i} & =\frac{\lambda}{N}\sum_{j=1}^{N}\psi_{ij}(x,v)(v_{j}-v_{i})
    \end{align*}
    where $N$ is the number of agents, $\psi_{ij}(x,v) \ge 0$ represents the interaction
	force, and $\lambda > 0$ is a constant. The system achieves flocking if
	$v_{i}(t)-v_{j}(t) \to 0$ for all $i,j$. The second equation can clearly be written in
	the form of (\ref{sys:cons}) and thus analyzed using the proposed methods.

	\item \textbf{Hegselmann--Krause model:} This model is widely used for opinion
			dynamics \cite{rainer2002opinion}. A time-varying version can be written as
    \[
    \dot{x}_{i}=\sum_{j=1}^{N}\phi_{ij}(t,x_{i},x_{j})(x_{i}-x_{j})
    \]
    where
    \[
    \phi_{ij}(t,x_{i},x_{j})=\begin{cases}
    1, & |x_{i}-x_{j}|\le\epsilon(t)\\
    0, & \text{otherwise}
    \end{cases}
    \]
    and $\epsilon:\mathbb{R}_{+}\to[0,1]$ is a measurable function.
	%This model can be
	%readily expressed in the form of (\ref{sys:cons}).

	\item \textbf{Animal group models:} The following is widely used to simulate animal
			group behavior (see, e.g., \cite{cristiani2011effects}):
    \[
			\hspace{-2mm} \dot{x}_{i}=\sum_{j\in A_{i}(t)}\frac{\phi_{a}(x_{i},x_{j})}{|x_{i}-x_{j}|}(x_{j}-x_{i})+\sum_{j\in R_{i}(t)}\frac{\phi_{r}(x_{i},x_{j})}{|x_{i}-x_{j}|}(x_{i}-x_{j})
    \]
	where $A_{i}$ and $\phi_{a}$ represent attraction, and $R_{i}$ and $\phi_{r}$
	represent repulsion. Both $\phi_{a}$ and $\phi_{r}$ are non-negative. This model 
	also has the form (\ref{sys:cons}).
\end{itemize}
}

The system \eqref{sys:cons} can also be written in matrix form as 
\begin{equation}
\dot{x}=A(t,x)x\label{sys:mat-form}
\end{equation}
in which $A(t,x)_{ij}=a_{ij}(t,x)$ for $i\ne j$ and $A(t,x)_{ii}= - \sum_{j\ne i}a_{ij}(t,x)$.
Note that $A(t,x)$ is Metzler and has the property $A(t,x)\mathds{1}=0$.

%\begin{remark}

%Throughout the paper, we assume that the solution to the system exists 
%and is unique for any given initial condition.
%We have restricted our attention to scalar agent dynamics . But the
%theory extends naturally to higher dimensions as long as it can be
%written in the form of \eqref{sys:cons} for scalar functions $a_{ij}(\cdot,\cdot)$.
%\end{remark}

Associated with the system \eqref{sys:mat-form} is a varying digraph $\mathcal{G}^{A(t,x)}$
understood in the following sense: for $i\ne j$, if $a_{ij}(t,x)>0$,
then $(i,j)$ is a directed link from node $i$ to $j$ and the weight
on this link is $a_{ij}(t,x)$; if $a_{ij}(t,x)=0$, then there is
no link from $i$ to $j$. Therefore, for any Metzler matrix $A$, there is an associated graph.
\begin{comment}
For convenience, we can use a square matrix $A(t,x)$ to represent
a graph by defining $A_{ij}(t,x):=a_{ij}(t,x)$ for all $(i,j)\in E(\mathcal{G}(t))$
and $A_{ij}(t,x)=0$ otherwise. In particular, $A(t,x)$ is a non-negative
matrix whose diagonal elements are zero.
\end{comment}
Note that, we do not consider self loop, i.e., the graph $\mathcal{G}^{A}$
is characterized only by the off-diagonal elements of $A$.
{The following definition collects a few important
notions that we will use frequently in the paper.}

{
\begin{definition}
	Let $A,B, C(z)\in \mathbb{R}^{n\times n}, \;z \in Z $
	be some Metzler matrices with row sum zero, $Z$ some index set,
	\begin{enumerate}
		\item We denote $\mathcal{G}^{A}$ 
			the graph associated with the matrix $A$.
		\item We say that $\mathcal{G}^{A} \ge \mathcal{G}^B $ if $A_{ij} \ge B_{ij}$ for all $i\ne
			j$.
		\item We say that the graph valued function $z\mapsto \mathcal{G}^{C(z)}$ is continuous if $z
			\mapsto C(z)$ is continuous. 
		\item The graph $\mathcal{G}^A$ is said to be quasi-strongly connected (QSC) if $\exists k
			\in \mathbb{N}_+$, such that $A_{ik}>0$ for all $i\ne k$; it is called
			$\delta$-connected for some $\delta>0$, if $\exists k\in
			\mathbb{N}_+$, such that $A_{ik} \ge \delta$ for all $i\ne k$.
\end{enumerate}
\end{definition}

}

In this note, consensus is understood in the following sense:
\begin{definition}
\label{def:consensus}Given a forward invariant set $D\subseteq\mathbb{R}^{n}$,
\begin{comment}
a forward invariant\footnote{The set $D$ is called forward invariant if $x_{i}(0)\in D$ for all
$i\in\{1,\cdots,n\}$ implies $x_{i}(t)\in D$ for all $i$ and $t\ge0$.} set $D\subseteq\mathbb{R}^{m}$,
\end{comment}
the system (\ref{sys:lin}) is said to achieve
\begin{itemize}
\item \emph{asymptotic consensus} on $D$ if there exists a class $\mathcal{KL}$
function $\beta$, such that $|x_{i}(t)-x_{j}(t)|\le\beta(|x_{i}(0)-x_{j}(0)|,t)$
for all $i,j\in\{1,\cdots,n\}$, $t\ge0$ and $x(0)\in D$;

\item \emph{exponential consensus} on $D$ if there exist some constants
$k,\lambda>0$ such that $|x_{i}(t)-x_{j}(t)|\le ke^{-\lambda t}|x_{i}(0)-x_{j}(0)|$
for all $i,j\in\{1,\cdots,n\}$, $t\ge0$ and $x(0)\in D$.
\end{itemize}
\end{definition}
\begin{comment}
We assume that $a_{ij}(t,x):\mathbb{R}_{+}\times\mathbb{R}^{n}\to\mathbb{R}_{+}$
are continuous in $x$ for all $(i,j)\in\mathcal{G}(t)$ and $h:\mathbb{R}\times\mathbb{R}_{+}\to\mathbb{R}$
is continuously differentiable. To study consensus of the system \eqref{sys:cons},
we introduce a notion called averaging graph.
\end{comment}
{} %
\begin{comment}
The non-negativity assumption on the coefficients $a_{ij}$ for $(i,j)\in E(\mathcal{G})$
is essential in our setting, which allows us to apply positive mapping
theory, or more precisely, the Hilbert metric.
\end{comment}

\begin{comment}
The model (\ref{sys:cons}) is quite general, at least for scalar
agent dynamics. For example, this includes the model $\dot{x}=A(t)x$
in Theorem \ref{thm:moreau} and the model in \cite{Lin2007} when
$h$ is set to zero. 
\end{comment}

\begin{remark}
For any $a<b$, $[a,b]^{n}$ is invariant under the system flow \eqref{sys:cons}.
Thus we may assume $D$ in Definition \ref{def:consensus} is compact. 

\end{remark}

The following regularity condition will be imposed on $A(t,x)$ throughout
the paper unless otherwise stated.
\begin{assumption}%[Regularity]
	{ Assume, for any compact set $D\subseteq \mathbb{R}^n$,
%\noindent \textbf{Assumption A1} (Regularity): 
%the mapping $x\mapsto A(t,x)x$ is locally Lipschitz continuous
%for every $t\ge0$; 2) given any compact set $D\subseteq\mathbb{R}^{n}$,
%there exists a constant $C_{D}$ such that $|A(t,x)|\le C_{D}$
%for all $x\in D$, $t\ge0$.%
			\begin{enumerate}
\item the mapping $t\mapsto A(t,x)$ is measurable for every fixed $x \in D$;
\item there is a locally integrable function $k_D(t)$ such that 
\begin{equation*}
		|A(t,x)x - A(t,y)y|\le k_D(t) |x-y|
\end{equation*}for all $x, y\in D, \; t\ge 0$.
\item there exists a constant $C_D>0$, such that $|A(t,x)|\le C_D$ for every $x \in D, \; t\ge 0$.
\end{enumerate}

%the mapping $x\mapsto A(t,x)x$ is locally Lipschitz continuous
%for every $t\ge0$; 2) given any compact set $D\subseteq\mathbb{R}^{n}$,
%there exists a constant $C_{D}$ such that $|A(t,x)|\le C_{D}$
%for all $x\in D$, $t\ge0$.%
}
\end{assumption}
{
		Under Assumption A1, the solution to system \eqref{sys:mat-form} exists and is
		unique through every point $(t_0, x_0) \in \mathbb{R}_+ \times D$ 
		\cite[Theorem 5.1, Theorem 5.3]{hale2009ordinary}.
}

\begin{comment}
\noindent 
\begin{enumerate}
\item locally Lipschitz continuous, uniformly in $t$, i.e., for any compact
set $D$, there exists a positive constant $L_{D}$, such that 
\begin{equation}
|a_{ij}(t,x)-a_{ij}(t,y)|\le L_{D}|x-y|\label{A1:LD}
\end{equation}
for all $x,y\in D$ and $t\ge0$;
\item locally bounded in the sense that for any compact set $D$, there
exists a constant $C_{D}>0$, such that 
\begin{equation}
|a_{ij}(t,x)|\le C_{D}\label{A1:Cd}
\end{equation}
for all $x\in D$ and $t\ge0$;
\item continuous, uniformly in $t$, i.e., for any $x$ and $\epsilon>0$,
there exists $\delta_{x}>0$ such that 
\[
|a_{ij}(t,x)-a_{ij}(t,y)|<\epsilon
\]
for all $|y-x|<\delta_{x}$ and $t\ge0$.
\end{enumerate}
\end{comment}

\begin{comment}
there exists a constant $T>0$ and a continuous hollow matrix (with
zero diagonal elements) $B(x)$, which represents a quasi-strongly
connected graph, such that 
\[
\text{off-diag}\int_{t_{0}}^{t_{0}+T}A(t,\phi(t;t_{0},x)){\rm d}t\ge B(x),\quad\forall t_{0}\ge0,\,\forall x.
\]
\end{comment}

\subsection{The Hilbert metric}

In this paper, we use Hilbert metric to analyze consensus problems.
Hilbert metric is a metric defined on cones. More precisely, in our
setting, a cone is some closed subset $K\subseteq\mathbb{R}^{n}$
satisfying the following four properties 1) The interior of $K$ is non-empty; 
2) For $v,w\in K$, $v+w$ is also in $K$. 3) For all $\lambda\ge0$, and $v\in K$, $\lambda v$ is also in $K$. 
4) $K\cap-K=\{0\}$, where $-K:=\{-x:x\in K\}$.
Given a cone, we can define a partial ordering
as $x\le y$ if $y-x\in K$ and $x<y$ if $y-x\in{\rm Int}\;K$. For
$x,y\in{\rm Int}K$, define two numbers $M(x/y)=\inf\{\lambda:x\le\lambda y\}$
or $\infty$ if the set is empty, and $m(x/y)=\sup\{\mu:\mu y\le x\}$.
Then the Hilbert metric between $x$, $y$ is defined as $d(x,y)=\ln\frac{M(x,y)}{m(x,y)}$.
Define the diameter of a set $S\subseteq K$ as ${\rm diam}(S)=\sup_{x,y\in S}d(x,y).$
%We allow ${\rm diam}S=+\infty$, e.g., $S=\mathbb{R}_{+}^{n}$. 
For any cone $S\subseteq{\rm Int}\mathbb{R}_{+}^{n}\cup\{0\}$,
we have ${\rm diam}S<+\infty$. In fact, $d(x,y)=\ln\frac{\max_{i}(x_{i}/y_{i})}{\min_{i}(x_{i}/y_{i})}$,
where $x=(x_{1},\cdots,x_{n})$, $y=(y_{1},\cdots,y_{n})$, which
is bounded on $S$. A mapping $A:K\to K$ on a cone is called non-negative.
If in addition, $A$ maps the interior of $K$ into its interior,
we call $A$ a positive mapping. For example, when $K=\mathbb{R}_{+}^{n}$,
then a non-negative matrix $A$ represents a non-negative mapping
while a positive matrix represents a positive mapping. 

%One of the most widely used results regarding positive mappings is
%Birkhoff's theorem 
%\begin{theorem}[Birkhoff \cite{birkhoff1957extensions}]
%Let $A$ be a positive mapping on $K$. If $A$ satisfies the following
%additional properties
%%
%%1) $A$ is homogeneous of degree $p>0$, i.e., $A(\lambda x)=\lambda^{p}A(x)$
%%for all $\lambda>0$.
%%
%%2) $A$ is monotone, i.e., $Ax\le Ay$ whenever $x\le y$;
%%
%Then $A$ is a contraction under the Hilbert metric in the following
%sense:
%\[
%d(Ax,Ay)\le kd(x,y)
%\]
%for some constant $k\in[0,1)$ for all $x,y\in K$. 
%\end{theorem}
%where the requirement of homogeneity is essential, see \cite{birkhoff1957extensions,Bushell1973}.
%For general nonlinear systems, however, it is no longer applicable.
%Therefore, new techniques need to be introduced. 

In the literature, positive mapping is often studied in the positive
orthant $\mathbb{R}_{+}^{n}$. The following simple example shows
that the positive orthant {may not} be the right cone for studying consensus.
Consider the system $\dot{x}_{1}=0$, $\dot{x}_{2}=x_{1}-x_{2}$.
It is then obvious that the system achieves exponential consensus.
The vector plot of the system is shown in Fig. \ref{fig:two-cones}.
We can see that the system does not contract the positive orthant
into its interior since $x_{1}=0$ is invariant. However, the system
does contract the smaller cone painted in gray. {This turns out to be a key
observation we need.} 
%Therefore, we propose
%to study the contraction of small cones in the positive orthant of
%the system \eqref{sys:cons}. %
\begin{comment}
Fig \ref{fig:two-cones} shows a linear system which does not contracts
the positive orthant, but it does contract any small cones contained
in the interior of the positive orthant. The same phenomenon appears
in nonlinear consensus. Once the sampled system is contractive with
respect to those cones under the Hilbert metric, we can conclude that
the system contracts to a fixed direction, which is in our case, the
span of $\mathds{1}_{n}$.
\end{comment}

\begin{comment}
Consider the following trivial system
\[
\dot{x}_{1}=0,\quad\dot{x}_{2}=x_{1}-x_{2}.
\]
\end{comment}
\begin{comment}
It is obvious that the system achieves exponential consensus, but
the system does not contract the cone $\mathbb{R}_{+}^{2}$. See the
left figure of Fig. However, the system contracts a smaller cone included
in ${\rm Int}\mathbb{R}_{+}^{2}\cup\{0\}$, as shown in the right
figure of Fig. \ref{fig:two-cones}. Thus, it is more reasonable to
restrict to a smaller cone in the positive orthant rather than the
whole of it. 
\end{comment}

\begin{figure}[th]
%\begin{centering}
		\centerline{\includegraphics[scale=0.39]{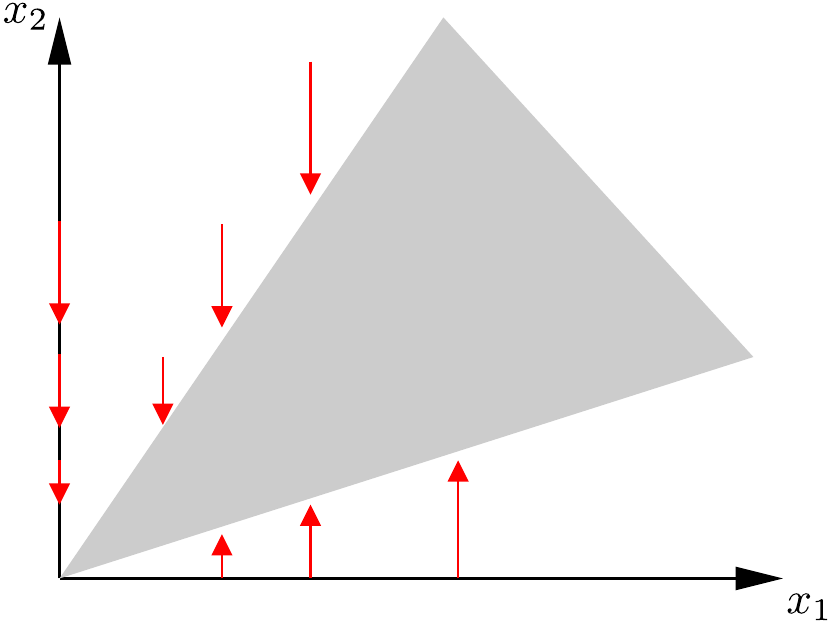}}
%\par
%\end{centering}
\caption{The red arrows represent the vector fields of a system. The $x_{2}$-axis
is invariant and hence cannot be mapped into the interior of the positive
orthant. However, the system contracts the smaller cone painted in
gray. \label{fig:two-cones}}
\end{figure}

{To proceed}, we need to justify that contraction in Hilbert metric
is equivalent to consensus that we defined earlier {in Definition
\ref{def:consensus}}. 
\begin{lemma}
\label{lem:cons-Hmetric}Suppose that $\mathcal{K}$ is a proper cone
in $\mathbb{R}_{+}^{n}$ satisfying $\mathcal{K}\subseteq\text{Int}\ \mathbb{R}_{+}^{n}\cup\{0\}$.
Then a system achieves asymptotic (resp. exponential) consensus on
$\mathcal{K}$ if and only if there exists class $\mathcal{KL}$ function
(resp. positive constants $K,\lambda$) such that 
\begin{align*}
		d(x{(t)},\mathds{1}) & \le\beta(d(x{(0)},\mathds{1}),t)\text{ (asymptotic)}\\
		d(x{(t)},\mathds{1}) & \le Ke^{-\lambda t}d(x{(0)},\mathds{1})\text{ (exponential)}
\end{align*}
where $d(x,y)$ stands for the Hilbert metric between $x$ and $y$.
\end{lemma}
\begin{proof}
We prove the asymptotic case -- the exponential case is similar.
Suppose that the system achieves {asymptotic} consensus on $\mathcal{K}$,
i.e., there exists a class $\mathcal{KL}$ function $\beta$ such
that $|x_{i}(t)-x_{j}(t)|\le\beta(|x_{i}(0)-x_{j}(0)|,t)$ for all
$t$ and $x(0)\in\mathcal{K}$ and $x(t)\in\mathcal{K}$ for all $t\ge0$.
{
Let
\[
A_{n}(x)=\left|x-\frac{|x|}{\sqrt{n}}\mathds{1}\right|^{2},\quad B_{n}(x)=\sum_{i,j=1}^{n}(x_{i}-x_{j})^{2};
\]
}
we estimate:
\begingroup
\allowdisplaybreaks
\begin{align*} 
d(x(t),\mathds{1})^{2} & =d\left(x(t),\frac{|x(t)|}{\sqrt{n}}\mathds{1}\right)^{2}\le\frac{1}{c_{1}}A_{n}(x(t))\text{ (Lemma \ref{lem:2norm-Hnorm}})\\
 & \le\frac{1}{c_{1}n}B_{n}(x(t))\text{ (Lemma \ref{lem:An-Bn})}\\
 & =\frac{1}{c_{1}n}\sum_{i,j=1}^{n}(x_{i}(t)-x_{j}(t))^{2}\\
 & \le\frac{1}{c_{1}n}\sum_{i,j=1}^{n}\beta(t,(x_{i}(0)-x_{j}(0))^{2})\\
 & \le\frac{n}{c_{1}}\beta(t,B_{n}(x(0)))\\
 & \le\frac{n}{c_{1}}\tilde{\beta}(t,d(x(0),\mathds{1}))\text{ (Lemma \ref{lem:An-Bn},\ref{lem:2norm-Hnorm})}
\end{align*}
\endgroup
in which $\tilde{\beta}$ is some class $\mathcal{KL}$ function.
The converse proceeds in a similar fashion.
\end{proof}
\begin{remark}
Although in Lemma \ref{lem:cons-Hmetric}, the initial condition is
restricted to a cone $\mathcal{K}$, in practice this is sufficient
for consensus on any compact set $D\subseteq\mathbb{R}^{n}$. Indeed,
let $y=x+\alpha\mathds{1}$. Then $\dot{y}=A(t,y-\alpha\mathds{1})y$
and consensus of $y$ is equivalent to consensus of $x$. By choosing
$\alpha>0$ sufficiently large, we may assume $D$ is in the interior
of $\mathbb{R}_{+}^{n}$. 
\end{remark}
We propose to study the following type of small cones. For $\gamma\in[0,\frac{1}{\sqrt{n}})$,
define a family of cones 
\begin{equation}
\mathcal{K}(\gamma):=\left\{ x\in\mathbb{R}^{n}:\frac{x_{i}}{|x|}\ge\frac{1}{\sqrt{n}}-\gamma,\;\forall i=1,\cdots n\right\}. \label{cone:K}
\end{equation}
See Fig. \ref{fig:coneke} for an illustration of such cones. The 
following lemma summarizes some properties of the diameter (in Hilbert metric)
of these cones.

\begin{figure}[!t]
		\centerline{\includegraphics[scale=0.45]{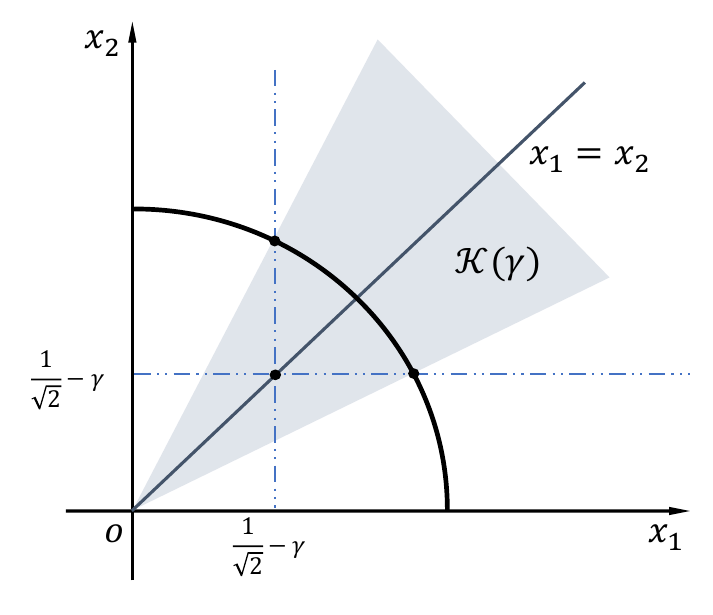}}
		\caption{Cone $\mathcal{K}(\gamma)$ when $n=2$. \label{fig:coneke}}
\end{figure}

\begin{lemma}
Given $\gamma\in[0,\frac{1}{\sqrt{n}})$, the diameter of $\mathcal{K}(\gamma)$
is 
\[
{\rm diam}\mathcal{K}(\gamma)=\log\left(1-n+\frac{n}{(1-\sqrt{n}\gamma)^{2}}\right)
\]
Denote $\alpha(\gamma)={\rm diam}\mathcal{K}(\gamma)$, then
\begin{itemize}
\item $\alpha$ is smooth, strictly increasing on $[0,1/\sqrt{n})$, 
and $\alpha(0)=0$, $\alpha(t)\to\infty$ as $t\to\frac{1}{\sqrt{n}}$.
The derivative of $\alpha$ is lower bounded away from zero.

\item For any $\epsilon_{0}\in[0,\frac{1}{\sqrt{n}})$ and $C\in(0,1)$,
there exist some positive constants $k_{1},k_{2}$ and $k\in(0,1)$,
such that $ k_{1}\gamma  \le\alpha(\gamma)\le k_{2}\gamma $,
and
\begin{equation}
\alpha(C\gamma)\le k\alpha(\gamma)\label{eq:contra-diam}
\end{equation}
for all $\gamma\in[0,\epsilon_{0}]$. %
\begin{comment}
As a result, $\alpha(C^{m}\gamma)\le k^{m}\alpha(\gamma)$ for all
$m\in\mathbb{Z}_{+}$.
\end{comment}
\end{itemize}
\end{lemma}
\begin{proof}
We show \eqref{eq:contra-diam}. Notice that
\begin{align*}
\alpha(r)-\alpha(Cr) & =\alpha'(\xi)(1-C)r\ge cr\ge\frac{c}{k_{2}C}\alpha(Cr)
\end{align*}
from which it follows that
\[
\alpha(Cr)\le\frac{1}{1+c/k_{2}C}\alpha(r).
\]
The rest is straightforward. 
\end{proof}
%
\begin{comment}
consider the function $\varphi(r)=\frac{\alpha(Cr)}{\alpha(r)}$ for
$r\in(0,\epsilon_{0}]$. Then $\varphi(r)<1$ for all $r\in(0,\epsilon_{0}]$.
Since $\varphi$ is continuous, then on any compact set $[\epsilon',\epsilon_{0}]$,
$\varphi$ is bounded away from $1$. On the interval $[0,\epsilon']$
where $\epsilon'$ is sufficiently small, define $\psi(r)=\alpha(r)-\alpha(Cr)$.
Then $\psi(0)=0$, $\psi'(0)>0$ and hence $\psi(r)>c'r$ for some
$c'>0$. That is, $\alpha(r)-\alpha(Cr)>c'r$
\[
\varphi(r)=\begin{cases}
1, & r=0\\
\frac{\alpha(Cr)}{\alpha(r)}, & r>0
\end{cases}.
\]
Then 
\end{comment}

\begin{comment}
The following simple comparison lemma will be useful to us. 
\begin{lemma}
\label{lem:exp-Hlbt}For any given $\epsilon_{0}\in(0,\frac{1}{\sqrt{n}})$,
there exist two positive constants, $c_{1},c_{2}$ (depending on $\epsilon_{0}$)
such that for all $\epsilon\in[0,\frac{1}{\sqrt{n}}-\epsilon_{0})$,
we have the following estimate of the diameter of the cone $\mathcal{K}(\epsilon)$:
\[
c_{1}\epsilon\le{\rm diam}(\mathcal{K}(\epsilon))\le c_{2}\epsilon,\quad\forall\epsilon\in[0,\frac{1}{\sqrt{n}}-\epsilon_{0})
\]
\end{lemma}
\begin{remark}
By Lemma \ref{lem:exp-Hlbt}, if $0<C<1$, we have 
\[
{\rm diam}\mathcal{K}(C^{m}\epsilon)\le\frac{c_{2}}{c_{1}}C^{m}{\rm diam}(\mathcal{K}(\epsilon))
\]
thus the cone $\mathcal{K}(C^{m}\epsilon)$ contracts to ${\rm span}\{\mathds{1}_{n}\}$
exponentially as $m\to\infty$. 
\end{remark}
\end{comment}

\subsection{Accumulated graph}

We have noted that in both \cite{Moreau2004} and \cite{Lin2007},
consensus is related to the accumulation of the graphs over time,
either the union of the switching graph in \cite{Lin2007}, or the
integration of the system matrix in \cite{Moreau2004}. This motivates
us to define the accumulated graph for a time-varying graph $\mathcal{G}(t)$,
$\forall t\ge0$. 

\begin{comment}
Let $\mathcal{G}(\cdot)=(a_{ij}(\cdot))$ be a measurable time-varying
graph, i.e., $t\mapsto a_{ij}(t)$ is a measurable function for all
$a_{ij}(\cdot)\in E(\mathcal{G}(t))$. The averaging graph on $\mathbb{R}_{\ge0}$
of $\mathcal{G}(\cdot)$ is defined as 
\begin{equation}
\bar{\mathcal{G}}=\lim\sup_{t\to\infty}\frac{1}{T}\int_{0}^{T}\mathcal{G}(t){\rm d}t\label{eq:av-graph}
\end{equation}
in which the integration is taken elementwisely.
\end{comment}

\begin{definition}[Accumulated graph]
Let $\mathcal{G}(\cdot)=(a_{ij}(\cdot))$ be a measurable time-varying
graph, i.e., $t\mapsto a_{ij}(t)$ is a measurable function for all
$i\ne j$. The accumulating graph of $\mathcal{G}(\cdot)$ over the interval
$[t_{1},t_{2}]$ is the graph $\mathcal{G}|_{t_{1}}^{t_{2}}$ defined
by the Lebesgue integral 
\[
\mathcal{G}|_{t_{1}}^{t_{2}}=\int_{t_{1}}^{t_{2}}\mathcal{G}(t){\rm d}t.
\]
in the sense that $(\mathcal{G}|_{t_{1}}^{t_{2}})_{ij}=\int_{t_{1}}^{t_{2}}a_{ij}(t){\rm d}t$. 
\end{definition}

\begin{example}
For a system (\ref{sys:lin}) with switching topology, the graph associated
with it can be written as $\mathcal{G}(t)=\sum_{i=1}^{N}1_{A_{i}}(t)\mathcal{G}_{i}$
where $1_{A_{i}}$ is the indicator function of some measurable sets
$A_{i}$ and $\mathcal{G}_{i}$ the graph corresponding to $p=i$.
Now the union graph on $[t,t+T]$ used in \cite{Lin2007} was nothing
but the accumulated graph $\mathcal{G}|_{t}^{t+T}$. This is quite
similar to the construction of the Lebesgue integration -- define
first for simple functions and then extend to larger class of functions. 
\end{example}
% In next section, we establish consensus results by studying the accumulated
% graph. Our proof strategy has a more geometric flavor based on Hilbert
% metric. %
With these technical preparations, we are now ready to present 
the main results of this paper.

\begin{comment}
It is this device that makes our extensions to more general graphs
possible.
\end{comment}

\section{Main results} \label{sec:mainresult}
We start with a lemma which explains how connectedness of the graph
is related to the contraction property under the Hilbert metric. The proof 
can be found in the Appendix.
\begin{lemma}
\label{lem:A-main}Let $A\in\mathbb{R}^{n\times n}$ be a non-negative
matrix and there exist a constant $\delta\in(0,1)$, and an integer
$k$ such that 
\begin{itemize}
\item $A\mathds{1}=\mathds{1}$;

\item $a_{ik}>\delta$ for all $i=1,\cdots,n$;
\end{itemize}
then for any $0<\epsilon<\frac{1}{\sqrt{n}}$, there exists a constant
$C\in(0,1)$ such that 
\begin{equation}
A\mathcal{K}(\epsilon)\subseteq\mathcal{K}(C\epsilon)\label{eq:lem:inc}
\end{equation} where $\mathcal{K}(\epsilon)$ is defined as in \eqref{cone:K}.
Moreover, $C$ can be taken as $\frac{1-\delta}{1-\sqrt{n}\epsilon\delta}$.
As a result, there exist a positive constant $c\in(0,1)$ such that
\begin{equation}
{\rm diam}(A^{m}\mathcal{K}(\epsilon))\le c^{m}{\rm diam}(\mathcal{K}(\epsilon)),\quad\forall m\ge1.\label{eq:lem:contrac}
\end{equation}
where $k$ can be taken as $c=\frac{1}{1+\eta C}$ and $\eta>0$ depends
on $\epsilon$, and $n$.
\end{lemma}
%\begin{remark} \label{rmk:delta-con}
%When a graph $A$ satisfies the assumptions of Lemma \ref{lem:A-main}, 
%we say that $\mathscr{G}(A)$ is $\delta$-connected.
    
%\end{remark}

%{
%\begin{definition}[Partial Ordering of Graphs]
		%For two multi-agent systems described by $\dot{x}_i = \sum_{j=1}^n a_{ij}(t,x) (x_i - x_j)$,
		%and $\dot{y}_i = \sum_{j=1}^n b_{ij}(t,y) (y_i - y_j)$, $i=1,\cdots,n$ where $a_{ij}(t,x)$ and 
		%$b_{ij}(t,y)$ are non-negative for all $(i,j)$. Let $\mathcal{G}_1(t,x)$ and $\mathcal{G}_2(t,x)$
		%be the graphs associated with the systems respectively. We denote 
		%\begin{equation}
				%\mathcal{G}_1 \ge \mathcal{G}_2 \text{ if } a_{ij}(t,x(t)) \ge b
		%\end{equation}
%\end{definition}

%}

Our first main result is the following theorem. 
\begin{theorem}
	\label{thm:main}{Let $D$ be a compact set in $\mathbb{R}^n$.} 
	Consider the multi-agent system (\ref{sys:mat-form}) satisfying Assumption \textbf{A1}. 
% Let $D\subseteq \mathbb{R}^n$ 
% be a compact invariant set and that there exists a constant $c_D$, such 
% that $|a_{ij}(t,x)|\le c_D$ for all $t\ge 0$, $x\in D$ and $i\ne j$. 
Let $\{t_k\}_{1}^\infty$ be an increasing sequence with $t_k \xrightarrow{
k\to \infty}\infty$ and $\sup_k |t_{k+1}-t_k|<\infty$.
If there exists a graph $\mathcal{G}^{B(x)}$ such that $B(x)$ is quasi-strongly connected and
{
\begin{equation} \label{eq:thm:low-bd}
\int_{t_k}^{t_{k+1}} \mathcal{G}^{A(t,\phi(t,t_k,x))} {\rm d} t
%\mathcal{G}(t,x) {\rm d}t
\ge \mathcal{G}^{B(x)}
\end{equation}
}for all $x\in D$, $k\ge 1$, then the system achieves
{\em asymptotic} consensus on $D$. If in addition, the graph 
{$\mathcal{G}^{B(x)}$} is continuous, the system achieves {\em exponential}
consensus on $D$.

% Then the system achieves {\em exponential} consensus on $D$ if 
% $x\mapsto a_{ij}(t,x)$ is locally Lipschitz for all $t\ge 0$, 
% and {\em asymptotic} consensus on $D$ if $x\mapsto A(t,x)$ is locally Lipschitz
% for all $t\ge 0$.
%\[
%\int_{t_{k}}^{t_{k+1}}\mathscr{G}(A(t,x)){\rm d}t\ge 
%\mathcal{G}(x)
%\]
%for all $k\ge 1$ and $x\in D$.
%Assume that $\mathcal{G}(t)$ satisfies Assumption \textbf{A2}.
%then
%the system achieves exponential consensus on $D$.
\end{theorem}
\begin{proof} 
We prove the second part first. 
As remarked earlier, by defining the coordinate transform
$y=x+\alpha\mathds{1}$ for $\alpha>0$ large, we may assume that
$D$ lies in a sufficiently small cone in $\mathbb{R}_{+}^{n}$. First,
assume $t\mapsto A(t,x)$ is piece-wise continuous. This assumption
will be removed later. Consider the Euler approximation scheme
on the interval $[t_{k},t_{k+1}]$:
\[
x_{i+1}=x_{i}+hA(t_{k}+ih,x_{i})x_{i},\quad i=0,\cdots,N-1
\]
with $h=\frac{t_{k+1}-t_{k}}{N}$, $N \in \mathbb{Z}_+$ large and
$x_{0}=x$. %
\begin{comment}
The following estimate is well known due to the Lipschitz continuity
of the mapping $x\mapsto A(t,x)x$ on any compact set $D$: 
\[
|x_{i}-\phi(t_{0}+ih,t_{0},x)|\le\frac{Ci}{N^{2}},\quad i\in\{0,\cdots,N\}.
\]
In particular, $|x_{N}-\phi(t_{0}+T,t_{0},x)|\le\frac{C}{N}$ where
$C$ depends only on the set $D$ and the constant $T$. 
\end{comment}
Choose $\lambda$ large enough such that $\bar{A}(t,x):=A(t,x)+\lambda I\ge0$
for all $t\ge0$ and $x\in D$. The following calculation is in order{\small{}{}
\begin{align}
x_{0} & =x\nonumber \\
x_{1} & =(1-h\lambda)x+h\bar{A}(t_{k},x)x\nonumber \\
x_{2} & =[(1-h\lambda)^{2}I+h(1-h\lambda)(\bar{A}(t_{k},x)+\bar{A}(t_{k}+h,x_{1}))+*]x\nonumber \\
 & \vdots\nonumber \\
x_{N} & =[(1-h\lambda)^{N}I+(1-h\lambda)^{N-1}\sum_{i=0}^{N-1}h\bar{A}(t_{k}+ih,x_{i})+*]x\label{eq:x_N}
\end{align}
}where $*$ stands for non-negative terms. %
\begin{comment}
Due to the compactness of $[t_{0},t_{0}+T]$, for any $\epsilon>0$,
there exists $\omega>0$ such that $|A(t,x)-A(t,y)|<\epsilon,\quad\forall|x-y|<\omega,\,\forall t\ge0$.
Choose $N$ large enough such that $\frac{C}{N}<\omega$, then 
\[
|\bar{A}(t_{0}+ih,\phi(t_{0}+ih,t_{0},x))-\bar{A}(t_{0}+ih,x_{i})|<\epsilon
\]
from which it follows that {\footnotesize{}{} 
\begin{equation}
\hspace{-6mm}\begin{aligned} & \left|\sum_{i=0}^{N-1}h\bar{A}(t_{0}+ih,x_{i})-\int_{t_{0}}^{t_{0}+T}\bar{A}(t,\phi(t,t_{0},x)){\rm d}t\right|\\
\quad\quad\le & \left|\sum_{i=0}^{N-1}h\bar{A}(t_{0}+ih,x_{i})-\sum_{i=0}^{N-1}h\bar{A}(t_{0}+ih,\phi(t_{0}+ih,t_{0},x))\right|\\
+ & \left|\sum_{i=0}^{N-1}h\bar{A}(t_{0}+ih,\phi(t_{0}+ih,t_{0},x))-\int_{t_{0}}^{t_{0}+T}\bar{A}(t,\phi(t,t_{0},x)){\rm d}t\right|\\
\le & \epsilon+\left|\sum_{i=0}^{N-1}h\bar{A}(t_{0}+ih,\phi(t_{0}+ih,t_{0},x))-\int_{t_{0}}^{t_{0}+T}\bar{A}(t,\phi(t,t_{0},x)){\rm d}t\right|
\end{aligned}
\end{equation}
} Since $\epsilon$ is arbitrary, we get 
\[
\lim_{N\to\infty}\sum_{i=0}^{N-1}h\bar{A}(t_{0}+ih,x_{i})=\int_{t_{0}}^{t_{0}+T}\bar{A(t},\phi(t,t_{0},x))dt
\]
because the Riemann sum of the piece-wise continuous function $t\mapsto\bar{A}(t,\phi(t;t_{0},x))$
converges to the integral. 
\end{comment}
Now let $P_{k}^{N}(x)$ be the matrix in the bracket on the right
hand side of (\ref{eq:x_N}) and define $P_{k}(x)=\lim_{{N}\to\infty}P_{k}^{N}(x)$.
Note that \begin{equation}
\phi(t_{k+1},t_{k},x)=P_{k}(x)x.\label{eq:Pk(x)}
\end{equation}
We claim that $P_{k}(x)$ has the following properties: 1) $P_{k}(x)$
is non-negative and $P_{k}(x)\mathds{1}=\mathds{1}$; 2) There exists
a non-negative matrix $S$ defining a QSC graph \footnote{Here we are slightly abusing the
concept of a QSC graph. It was previously defined as the graph associated with a continuous-time system
\ref{sys:mat-form}. Here, it
should be understood in the sense that there exists some $k\in \mathbb{N}_+$, such that $S_{ik}>0$ for all
$i\ne k$.
}, independent of $k,x$
such that $P_{k}(x)\ge S$. Item 1) is obvious, we verify 2). From
\eqref{eq:x_N} we see 
\[
P_{k}(x)\ge e^{-\lambda(t_{k+1}-t_{k})}\left(I+\int_{t_{k}}^{t_{k+1}}\bar{A}(t,\phi(t,t_{k},x)){\rm d}t\right)\ge S
\]for some non-negative $S$ such that {$\mathcal{G}^S$} is QSC
thanks to \eqref{eq:thm:low-bd},
the continuity of $\mathcal{G}^{B(x)}$ and compactness of $D$. Without loss of generality, we may assume
that $\mathcal{G}^S$ is $\delta$-connected (since
otherwise, we can consider the matrix $P_{k+n}P_{k+n-1}\cdots P_{k}$,
which will be lower-bounded by $S^{n}$ that is $\delta$-connected).
Invoking Lemma \ref{lem:A-main}, we conclude that the system achieves
exponential consensus on a cone of the form \eqref{cone:K} which includes
$D$ in its interior.

To remove the assumption on piece-wise continuity of $t\mapsto A(t,x)$
on $[t_{k},t_{k+1}]$, it suffices to replace Riemann integration
by Lebesgue integration as follows. View $\varphi:t\mapsto A(t,\cdot)$
as a mapping from $[t_{k},t_{k+1}]$ to $C(D;\mathbb{R}^{n\times n})$
equipped with norm $||g||_{*}=\sup_{x\in D}||g(x)||$. 
Then $\varphi\in L^{\infty}([t_{k},t_{k+1}];C(D;\mathbb{R}^{n\times n}))$,
which can be approximated by simple functions. Let $\eta>0$ be an
arbitrarily small constant, and $\mathscr{A}(t,x)=\sum_{i=0}^{N-1}1_{[s_{i},s_{i+1})}(t)A_{i}(x)$
the simple function such that $\|\varphi-\mathscr{A}\|_{\infty}<\eta$.
Assume that the partition is uniform (otherwise we can always refine the 
partition to make it close to uniform and then use approximation arguments),
i.e., $s_{i+1}-s_{i}=\frac{t_{k+1}-t_{k}}{N}$
for all $i$. Let $\bar{A}_{i}(x)=A_{i}(x)+\lambda I\ge0$, $\bar{\mathscr{A}}(t,x)=\mathscr{A}(t,x)+\lambda I$.
As before, the Euler approximation scheme gives \eqref{eq:x_N}.
Now
\begin{align*}
\sum_{i=0}^{N-1}h\bar{A}(t_{k} & +ih,x_{i})=\sum_{i=0}^{N-1}h\bar{A}_{i}(x_{i})\\
 & =\int_{t_{k}}^{t_{k+1}}\bar{\mathscr{A}}(t,\phi(t,t_{k},x)){\rm d}t\\
 & \ge\int_{t_{k}}^{t_{k+1}}\bar{A}(t,\phi(t,t_{k},x)){\rm d}t-\eta(t_{k+1}-t_{k})\mathds{1}_{n\times n}.
\end{align*}
Since $|t_{k+1}-t_{k}|$ is uniformly bounded, we can choose $\eta$
sufficiently small such that $P_{k}^{N}(x)$ is lower-bounded by a
QSC graph which is independent of $k$. Thus the system achieves exponential
consensus on $D$.

It remains to prove asymptotic consensus when $\mathcal{G}^{B(x)}$ 
is not necessarily continuous. As before, we can assume
that $t\mapsto A(t,x)$ is piecewise continuous. The Euler approximation
still converges due to local Lipschitz continuity of the 
system vector fields. Now, instead of
having a uniform lower bound on $P_{k}(x)$ (see \eqref{eq:Pk(x)}),
we only know that it is bounded by some QSC graph. 
But still, we know that for small $\epsilon_{0}>0$,
$P_{k}(x)\mathcal{K}(\epsilon_{0})\subseteq\mathcal{K}(\epsilon_{1})$.
% or ${\rm diam}\phi(t_{k+1},t_{k},S)<{\rm diam}S$ for any $S\subseteq\mathcal{K}(\epsilon_{0})$.
Since $D$ is compact, we may assume that $D\subseteq\mathcal{K}(\epsilon_{0})$.
Define $D_{1}=\phi(t_{1},t_{0},D)$, $D_{k+1}=\phi(t_{k+1},t_{k},D_{k})$
for $k\ge1$ and $\eta_{k}={\rm diam}D_{k}$ which is a 
non-negative decreasing sequence. Set a decreasing sequence
$\{\epsilon_{k}\}_{1}^{\infty}$, such that $D_{k}\subseteq\mathcal{K}(\epsilon_{k})$.
Clearly, $\eta_{k}$ is strictly decreasing and is bounded from
below by $0$. Therefore $\eta_{k}$ converges to a limit $c\ge0$.
\begin{comment}
for all $t_{0}\in\mathbb{R}$ and the fixed $T>0$. Thus $P(t_{0}+T,t_{0},\cdot)$
contracts every cone $\mathcal{K}(\epsilon)$. Let $D={\rm co}\{x_{1},\cdots,x_{n}\}\subseteq\mathcal{K}(\epsilon)$,
by \textbf{A2}, there exists $\epsilon_{1}<\epsilon$ such that 
\[
\min_{x\in D}\frac{\phi(t_{0}+T,t_{0},x)_{i}}{|\phi(t_{0}+T,t_{0},x)|}=\frac{1}{\sqrt{n}}-\epsilon_{1}
\]
(the minimum is attained since $\phi$ is continuous). By induction,
we can define a strictly decreasing sequence $\{\epsilon_{m}\}_{m=1}^{\infty}$,
for example, 
\begin{align*}
\frac{1}{\sqrt{n}}-\epsilon_{2} & =\min_{x\in D}\frac{\phi(t_{0}+2T,t_{0},x)_{i}}{|\phi(t_{0}+2T,t_{0},x)|}\\
 & =\min_{x\in\phi(t_{0}+T,t_{0},D)}\frac{\phi(t_{0}+2T,t_{0}+T,x)_{i}}{|\phi(t_{0}+2T,t_{0}+T,x)|}
\end{align*}
with $\epsilon_{2}<\epsilon_{1}$ since $\phi(t_{0}+T,t_{0},D)\subseteq\mathcal{K}(\epsilon_{1})$.
Then $\epsilon_{m}\to c\ge0$ as $m\to\infty$.
\end{comment}
 
\begin{figure}[!t]
		\centerline{\includegraphics[scale=0.5]{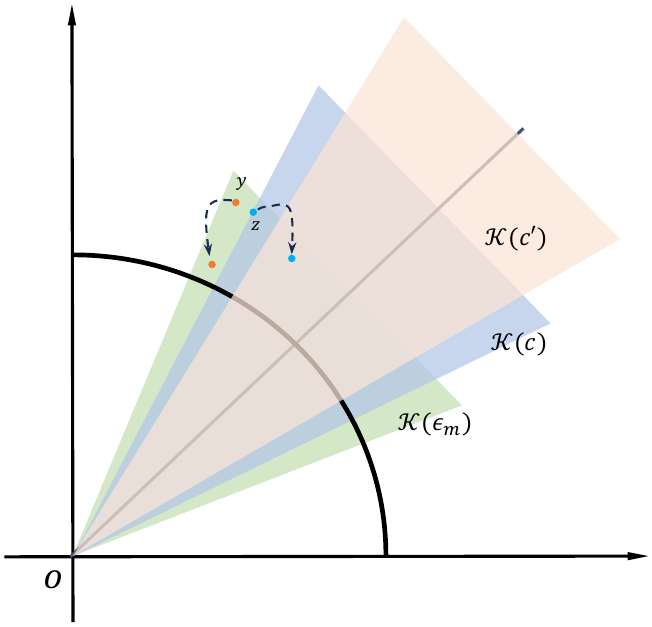}}
\caption{Asymptotic contraction\label{fig:asymp}.}
\end{figure}

Suppose $c>0$, and choose $\eta_{m}$ sufficiently close to $c$.
There must exist $y\in\mathcal{K}(\epsilon_{m})\backslash\mathcal{K}(c)$,
such that $\phi(t_{m+1},t_{m},y)\in\mathcal{K}(\epsilon_{m})\backslash\mathcal{K}(c)$.
Choose $z\in\mathcal{K}(c)$ close to $y$, then due to the continuity
of $y\mapsto\frac{\phi(t,t_{0},y)}{|\phi(t,t_{0},y)|}$, the error
between $\frac{\phi(t_{m+1},t_{m},y)}{|\phi(t_{m+1},t_{m},y)|}$ and
$\frac{\phi(t_{m+1},t_{m},z)}{|\phi(t_{m+1},t_{m},z)|}$ is of order
$O(|\epsilon_{m}-c|)$ which can be made sufficiently small by choosing
$m$ large enough. But {\small{}
\[
\min_{w\in\mathcal{K}(c)\cap D}\frac{\phi(t_{m+1},t_{m},w))_{i}}{|\phi(t_{m+1},t_{m},w)|}=\frac{1}{\sqrt{n}}-c'>\frac{1}{\sqrt{n}}-c
\]
(since $\eta_{m+1}<\eta_{m}$) }implying that $\frac{\phi(t_{m+1},t_{m},z)_{i}}{|\phi(t_{m+1},t_{m},z)|}$
is away from $\frac{1}{\sqrt{n}}-c$ for any $m$. This is a contradiction
since $\frac{\phi(t_{m+1},t_{m},y)}{|\phi(t_{m+1},t_{m},y)|}\in\mathcal{K}(\epsilon_{m})\backslash\mathcal{K}(c)$
and that $\frac{\phi(t_{m+1},t_{m},y)}{|\phi(t_{m+1},t_{m},y)|}$
and $\frac{\phi(t_{m+1},t_{m},z)}{|\phi(t_{m+1},t_{m},z)|}$ are sufficiently
close. Thus we conclude that $c=0$ and asymptotic consensus is achieved.
The proof strategy is shown in Fig. \ref{fig:asymp}.
\begin{comment}

\subsection{Higher order agent dynamics}

The result can be easily extended to non scalar agent dynamics as
in \cite{Lin2007}. There, $A(t,x)$ has structure $A(t,x)=\tilde{A}(t,x)\otimes I_{m}$
where $\tilde{A}(t,x)\in\mathbb{R}^{n\times n}$. From previous calculation,
we see that $\phi_{A}(t,t_{0},x)$ can be decomposed as $P(t,t_{0},x)\otimes I_{m}$
where $P(t,t_{0},x)$ converges to some one rank matrix $\mathds{1}_{n}\varphi(x)^{\top}$,
where $\varphi(x)\in\mathbb{R}^{n}$. Thus the system states converge
toward $\mathds{1}\otimes\alpha$ exponentially for some $\alpha\in\mathbb{R}^{m}$.
More generally, one can even consider consensus on manifolds. Let
$(M,g)$ be a Riemannian manifold. Denote $\log_{x}:M\to T_{x}M$
as the inverse mapping of the exponential map $\exp_{x}:T_{x}M\to M$
and consider the following multi-agent system: 
\[
\dot{x}_{i}=f(x_{i})+\sum_{j=1}^{n}a_{ij}(t,x)\log_{x_{i}}x_{j},\quad i=1,\cdots,n
\]
where $a_{ij}(t,x):\mathbb{R}_{+}\times M\to\mathbb{R}_{+}$ are non-negative
continuous functions on $M$. Although the agent is no longer one
dimensional scalar dynamics, its consensus properties is exactly the
same as those in Euclidean space.
\end{comment}

The proof is now complete.
%
\begin{comment}
{\small{}
\begin{align*}
\sum_{i=0}^{N-1} & h\bar{A}(x_{i})\ge\int_{t_{k}}^{t_{k+1}}(\mathscr{A}(t,\phi(t,t_{0},x))+\lambda I){\rm d}t+\epsilon\mathds{1}_{n\times n}\\
 & =\int_{t_{k}}^{t_{k+1}}\bar{A}(t,\phi(t,t_{0},x)){\rm d}t\\
 & \;+\int_{t_{k}}^{t_{k+1}}(\mathscr{A}(t,\phi(t,t_{0},x))-\bar{A}(t,\phi(t,t_{0},x))){\rm d}t+\epsilon\mathds{1}_{n\times n}\\
 & \ge\int_{t_{0}}^{t_{0}+T}\bar{A}(t,\phi(t,t_{0},x)){\rm d}t-(\epsilon+\eta)\mathds{1}_{n\times n}
\end{align*}
}where the constant $\epsilon>0$ represents an infinitesimal error
caused by the error between $\phi(t_{k+1},t_{k},x)$ and $x_{N}$.
Let 
\[
Q_{k}^{N}(x)=[(1-h\lambda)^{N}I+(1-h\lambda)^{N-1}\sum_{i=0}^{N-1}h\bar{A}(x_{i})+*]
\]
. Since $(1-h\lambda)^{N-1}$ is bounded away from $0$ for all $N$,
we can choose $\eta$ sufficiently small such that $(P_{N}(x))_{ik}>\delta>0$
for all $i$. The rest of the proof is the same as in the piece-wise
continuous case. Thus the theorem holds.
\end{comment}
{} \end{proof}

Theorem \ref{thm:main} needs the computation of the integral $\int_{t_{k}}^{t_{k+1}}A(t,\phi(t,t_{k},x)){\rm d}t$,
which is impossible in most cases -- except that $A(t,x)$ does not
dependent on $x$. The following corollary is more convenient for
practical use. 
% \end{remark}
\begin{corollary} \label{cor:main}
Consider the system \eqref{sys:cons} under Assumption \textbf{A1}. 
Let $D$ be a compact invariant set.
Suppose that there exists a
(continuous) QSC graph {$\mathcal{G}^{B(x)}$,} continuous on
$D$, and an increasing
sequence $\{t_{k}\}_{1}^{\infty}$ with $t_{k}\xrightarrow{k\to\infty}\infty$
and $\sup|t_{k+1}-t_{k}|<\infty$, such that 
{
\[
	\int_{t_{k}}^{t_{k+1}}\mathcal{G}^{A(t,x)}{\rm d}t\ge \mathcal{G}^{B(x)},\quad\forall x\in D,\;k\ge1
\]
} then the system achieves (asymptotic) exponential consensus on $D$.
\end{corollary}
\begin{proof}

%Denote $\mathcal{G}(x)=\mathcal{G}^{B(x)}$.
		Assume $B(x)$ is continuous.
We utilize Euler approximation as before. Note that for fixed $r\ne s\in\{1,\cdots,n\}$,
there exists $x_{rs}\in D$ such that 
\[
\sum_{i=1}^{N-1}ha_{rs}(t_{k}+hi,\phi(t_{k}+hi,t_{k},x))\ge\sum_{i=1}^{N-1}ha_{rs}(t_{k}+hi,x_{rs})
\]
since for every $k$ and $i$, $x\mapsto a_{rs}(t_{k}+hi,\phi(t_{k}+hi,t_{k},x))$
is continuous. But the right hand side is an approximation of $\int_{t_{k}}^{t_{k+1}}a_{rs}(t,x_{rs}){\rm d}t$,
which is lower bounded by $B_{rs}(x_{rs})$. Define a matrix $\tilde{B}:=(B_{rs}(x_{rs}))$.
Due to the continuity of $x\mapsto B(x)$, $B(x)$ and hence $\tilde{B}$,
are lower bounded by some matrix $S$ such that $\mathcal{G}^{S}$
is QSC. The conclusion follows invoking Theorem \ref{thm:main}.
\end{proof}
\begin{example}
% Let us check that Theorem \ref{thm:moreau} and \ref{thm:lin} are direct
% consequences of Theorem \ref{thm:main} or Corollary \ref{cor:main}.
1) Theorem \ref{thm:moreau} is now a corollary of Theorem \ref{thm:main}.
Our result is slightly stronger: the mapping $t\mapsto A(t)$ is only
required to be bounded measurable while in Theorem \ref{thm:moreau},
this mapping is assumed to be piecewise continuous. Note that the proof
strategy is rather different for the two theorems.

2) For Theorem \ref{thm:lin}, the system \eqref{sys:lin}
can be written as
\[
\dot{x}=\left[\sum_{k=1}^{p}1_{k}(\sigma(t))A_{k}(x)\right]x.
\]
Fix an interval
$[t_{0},t_{0}+T]$, then on $[t_{0}-\tau_{D},t_{0}+T+\tau_{D}]$,
we integrate 
\[
\int_{t_{0}-\tau_{D}}^{t_{0}+T+\tau_{D}}\sum_{k=1}^{p}1_{k}(\sigma(t))A_{k}(x){\rm d}t\ge\tau_{D}\sum_{k=1}^{p}A_{k}(x).
\]
By assumption, $\sum_{k=1}^{p}A_{k}(x)$ is QSC and hence the system
achieves asymptotic consensus. 

If we assume further that $A_{k}(x)$ is continuous, then we get exponential
consensus from Theorem \ref{thm:main}. But as far 
as we know, it is not clear how to prove this using the techniques 
in \cite{Lin2007}.
% The reason why only asymptotic consensus
% was obtained in \cite{Lin2007} was due to weaker regularity assumptions
% therein. In Theorem \ref{thm:lin}, regularity assumptions are imposed
% on the function $x\mapsto A_{k}(x)x$ but here we need continuity
% assumptions on $x\mapsto A_{k}(x)$. However, it is not clear how
% to show exponential consensus under this stronger condition using
% the method in \cite{Lin2007}.
In addition, Theorem \ref{thm:lin}
was proven only for switching multi-agent system with regular switching
signal while Theorem \ref{thm:main} only requires the ``switchings''
to be measurable which is always satisfied in practice.
\end{example}
\begin{example} \label{exmp:kuramoto}
Consider a Kuramoto model with identical frequency 
\begin{equation}
\dot{\theta}_{i}=\omega+\sum_{j\in\mathcal{N}_{t}(i)}a_{ij}(t)\sin(\theta_{j}-\theta_{i}),\quad i=1,\cdots n\label{eq:kuramoto}
\end{equation}
where $\mathcal{N}_{t}(i)$ stands for the neighboring node of $i$
at time $t$ and $a_{ij}(\cdot)\in L^{\infty}(\mathbb{R}_{\ge0},\mathbb{R}_{>0})$.
Let $a,b$ be real numbers such that $0\le b-a<\pi$ and 
the graph $ \mathcal{G}^{A(t)}$ satisfies 
the assumption of Theorem \ref{thm:main} where the lower
bound of the accumulated graph is now state-independent.
% the graph $\mathcal{is satisfied for $\mathcal{G}_{{\rm K}}$.

We claim that the system
\eqref{eq:kuramoto} achieves exponential consensus on $[a,b]^{n}$.
In particular, if $\dot{\theta}_{i}=\omega+\sum_{j\in\mathcal{N}(i)}\sin(\theta_{j}-\theta_{i})$
and the graph associated with the system is QSC, then exponential
consensus is achieved on $[a,b]^{n}$. To see the claim, let $x_{i}=\theta_{i}-\omega$,
the above model can be written as
\[
\dot{x}_{i}=\sum_{j\in\mathcal{N}_{t}(i)}a_{ij}(t)\sin(x_{j}-x_{i})
\]
or in matrix form $\dot{x}=A(t,x)x$ where $(A(t,x))_{ij}=\frac{\sin(x_{j}-x_{i})}{x_{j}-x_{i}}a_{ij}(t)$
for $i\ne j$ and $j\in\mathcal{N}_{t}(i)$. Then on $[a,b]^{n}$,
$(A(t,x))_{ij}\ge ca_{ij}(t)$ for some positive constant $c$. Thus
consensus is determined by the graph
$\mathcal{G}^{A(t)}$.

%This extends \cite{Lin2007}, where only asymptotic synchronization
%is obtained. However, this result can be proven using other methods,
%see e.g. \cite{dorfler2014synchronization}.
\end{example}
\begin{comment}

\subsection{Comparison result}
\begin{corollary}
Consider the system (\ref{sys:cons}) and assume \textbf{A1}. If there
exists a QSC graph $B(x)$ such $A(t,x)\ge B(x)$ for all $t$, then
the system achieves exponential consensus on any compact convex set. 
\end{corollary}
\begin{proof} Obvious since $\int_{t}^{t+T}A(t,x)dt>TB(x)$, for
any $t$ and $T$. Note that it is not clear how to obtain exponential
from \cite{Lin2007}. \end{proof} 
\end{comment}
%Note that in (\ref{sys:lin}),
%only asymptotic synchronization can be obtained. 
%\vspace{-2mm}
\section{Simulation Result} \label{sec:simu}
We simulate Example \ref{exmp:kuramoto} for
$\omega =0$. Consider a chain structure
as in Fig. \ref{fig:chain} of
$N=10$ oscillators. The weights on the link $(x_{i} \to x_{i+1})$ is 
$a_{i,i+1}(t)$. The weights are generated in the following manner.
First, generate some random intervals $[t_k, t_{k+1}]$ for $k=1,\cdots ,M$
for some large $M$. Then, divide each $[t_k, t_{k+1}]$
into $N$ smaller pieces $[s_i^k, s^k_{i+1}]$ randomly. After that,
Euler scheme will run on each interval $[s_i^k,s^k_{i+1}]$. In each
step of the Euler scheme, we choose three random 
numbers $p,q,r$ from $\{1,\cdots ,N-1\}$ and generate three
random positives numbers $a_{p,p+1}$, 
$a_{q,q+1}$, $a_{r,r+1}$ lower bounded by
a threshold $\delta>0$ and the rest of $a_{i,i+1}$ are set to zero.

By construction, the weights $a_{i,i+1}$ are zero for most of the time
and since the these weights are generated randomly, they are 
quite irregular. But still, we can see from
Fig. \ref{fig:kuramoto} that the system achieves consensus.

\begin{figure}[ht]
\centerline{
\includegraphics[scale=0.4]{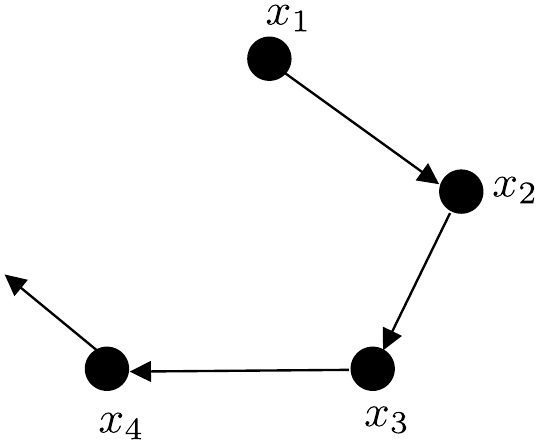}
}
\caption{A chain of coupled oscillators. \label{fig:chain}}
\end{figure}

\begin{figure}[ht]
\centerline{
\includegraphics[scale=0.35]{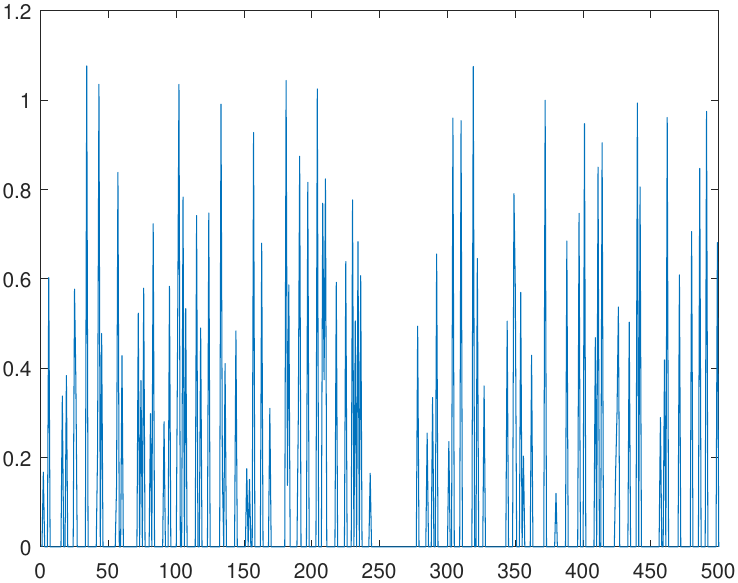}
\includegraphics[scale=0.35]{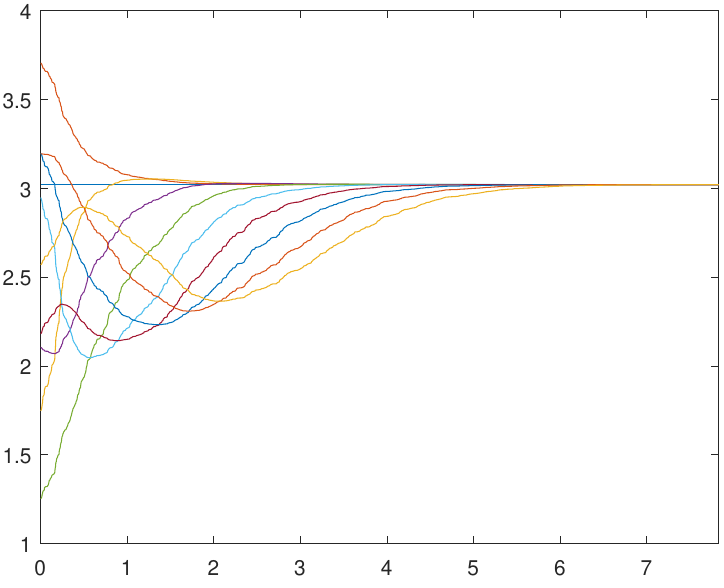}
}
\caption{Simulation result. Left: the signal $a_{1,1+1}(t)$. Right: the state components. \label{fig:kuramoto} }
\end{figure}

\section{Conclusion}

In this technical note, we have shown that the Hilbert metric can
serve as alternative tool to study consensus properties. It is advantageous
in dealing with nonlinearities and time dependencies, and requires
very weak regularity assumptions. The results obtained in this note
are somewhat preliminary and open the door for future research. 

\begin{comment}
The current framework is quite flexible. It can be further studied
in several different directions: 1) The internal dynamics can be nonlinear.
2) Stochastic noise can be introduced, e.g., randomness in the interaction
forces. In that case, probability estimations of consensus can be
derived. 3) One can consider more general accumulated graphs rather
than QSC ones. Then we will be able to analyze more general phenomena
than consensus, e.g., multi-synchronization.
\end{comment}

\section{Appendix}
\begin{proof}[Proof of Lemma \ref{lem:A-main} ]
Consider the set
\[
\mathcal{E}(\epsilon)=\left\{ e:\frac{1}{\sqrt{n}}-\epsilon\le e_{i}\le\frac{1}{\sqrt{n}},\;\forall i=1,\cdots,n,\;|e|\le1\right\} .
\]
Then $\mathcal{K}(\epsilon)=\{\alpha e:\forall\alpha\in\mathbb{R}_{>0},\;e\in\mathcal{E}(\epsilon)\}$.
Thus it is sufficient to prove
\begin{equation}
\frac{(Ae)_{i}}{|Ae|}\ge\frac{1}{\sqrt{n}}-C\epsilon,\quad\forall i\ge1,\;e\in\mathcal{E}(\epsilon).\label{eq:lem-proof}
\end{equation}
Write $e=\frac{1}{\sqrt{n}}-\tilde{e}$, then $0\le\tilde{e}_{i}\le\epsilon$
and $(Ae)_{i}=\frac{1}{\sqrt{n}}-(A\tilde{e})_{i}$. On the one
hand, $(A\tilde{e})_{i}=\sum_{j\ne k}^{n}a_{ij}\tilde{e}_{j}+a_{ik}\tilde{e}_{k}\ge\delta\tilde{e}_{k}$.
On the other hand,
\begin{align*}
\sum_{j\ne k}^{n}a_{ij}\tilde{e}_{j}+a_{ik}\tilde{e}_{k} & \le\epsilon\sum_{j\ne k}^{n}a_{ij}+a_{ik}\tilde{e}_{k}\\
 & =\epsilon(1-a_{ik})+a_{ik}\tilde{e}_{k}\\
 & \le\delta(\tilde{e}_{k}-\epsilon)+\epsilon .
\end{align*}
Thus we obtain the inequality
\[
\frac{1 }{\sqrt{n}}-\epsilon(1 -\delta)-\delta\tilde{e}_{k}\le(Ae)_{i}\le\frac{1 }{\sqrt{n}}-\delta\tilde{e}_{k}.
\]
As a result $|Ae|\le1 -\sqrt{n}\delta\tilde{e}_{k}$ and
\[
\frac{(Ae)_{i}}{|Ae|}\ge\frac{\frac{1 }{\sqrt{n}}-\delta\tilde{e}_{k}-\epsilon(1 -\delta)}{1 -\sqrt{n}\delta\tilde{e}_{k}}=\frac{1}{\sqrt{n}}-\frac{1 -\delta}{1 -\sqrt{n}\tilde{e}_{k}\delta}\epsilon
\]
Recall that $\tilde{e}_{k}\le\epsilon<\frac{1}{\sqrt{n}}$, we get
$\frac{1 -\delta}{1 -\sqrt{n}\tilde{e}_{k}\delta}\le\frac{1 -\delta}{1 -\sqrt{n}\epsilon\delta}:=C\in(0,1)$
and \eqref{eq:lem-proof} follows. In other words, we have shown $A\mathcal{K}(\epsilon)\subseteq\mathcal{K}(C\epsilon)$.

To show \eqref{eq:lem:contrac}, following the same procedure, we
can prove
$
A\mathcal{K}(C\epsilon)\subseteq\mathcal{K}(C'C\epsilon)
$,
where $C'=\frac{1 -\delta}{1 -\sqrt{n}C\epsilon\delta}<C$.
Thus $A^{2}\mathcal{K}(\epsilon)\subseteq A\mathcal{K}(C\epsilon)\subseteq\mathcal{K}(C^{2}\epsilon)$.
By induction, we can show $A^{m}\mathcal{K}(\epsilon)\subseteq\mathcal{K}(C^{m}\epsilon)$.
Thus the inequality \eqref{eq:lem:contrac}.
\end{proof}
\begin{lemma}
\label{lem:An-Bn}For $x\ge0$, define two quantities as
\[
A_{n}(x)=\left|x-\frac{|x|}{\sqrt{n}}\mathds{1}\right|^{2},\quad B_{n}(x)=\sum_{i,j=1}^{n}(x_{i}-x_{j})^{2}
\]
then $nA_{n}\le B_{n}\le2nA_{n}$.
\end{lemma}
\begin{proof}
It suffices to prove for $|x|=1$. Define $X=\sum_{i=1}^{n}x_{i}$,
then
\[
A_{n}=2-\frac{2}{\sqrt{n}}X,\quad B_{n}=2n-2X^{2}.
\]
Since $|x|=1$, we have $0\le X\le\sqrt{n}$. Thus $B_{n}\ge2n-2\sqrt{n}X=nA_{n}$.
On the other hand, $2nA_{n}-B_{n}=2(X-\sqrt{n})^{2}\ge0$.
\end{proof}
\begin{lemma}
\label{lem:2norm-Hnorm}Consider the cone $\mathbb{R}_{+}^{n}$. Suppose
that $\mathcal{K}$ is a proper cone such that $\mathcal{K}\subseteq\text{Int}\ \mathbb{R}_{+}^{n}\cup\{0\}$.
Then there exist some constants $C_{2}>C_{1}>0$ such that $C_{1}d(x,w)\le|x-w|\le C_{2}d(x,w)$
for all $x\in\mathbb{R}_{+}^{n},$ $w\in\mathcal{K}$ with $|w|=|x|=1$. 
\end{lemma}
\begin{proof} Since $|w|=|x|=1$, we can find a constant $C>0$,
such that 
\[
C\tanh\left(\frac{1}{2}d(x,w)\right)\le|x-w|\le\exp\left(d(x,w)\right)-1
\]
invoking \cite[Theorem 4.1]{Bushell1973}, see \cite[(4.1) and (4.3)]{Bushell1973}.
The conclusion follows by noticing that $d(x,w)$ is bounded for all
$x\in\mathbb{R}_{+}^{n}$, $w\in\mathcal{K}$ with $|w|=|x|=1$.
\end{proof}

\bibliographystyle{IEEEtran}
\bibliography{mono}
%\bibliography{mainv2}

\end{document}